\documentclass[12pt]{article}
\usepackage[margin=2cm]{geometry}
\usepackage{amsmath,amstext,amsfonts,amsthm,bbm}

\theoremstyle{plain}

\newtheorem{prop}{Proposition}
\newtheorem{defin}{Definition}[section]

\newtheorem{lemma}[defin]{Lemma}
\newtheorem*{thm*}{Theorem}
\newtheorem*{prop*}{Proposition}

\theoremstyle{remark}
\newtheorem{rem}[defin]{Remark}

\DeclareRobustCommand{\stirling}{\genfrac\{\}{0pt}{}}

\begin{document}
\title{On the number of zeros of functions in analytic quasianalytic classes}
\author{Sasha Sodin\footnote{School of Mathematical Sciences, Queen Mary University of London, London E1 4NS, United Kingdom \& School of Mathematical Sciences, Tel Aviv University, Tel Aviv, 69978, Israel.  Email: a.sodin@qmul.ac.uk. This work is supported in part by the European Research Council starting grant 639305 (SPECTRUM) and by a Royal Society Wolfson Research Merit Award.}}
\maketitle

\begin{abstract}
A space of analytic functions in the unit disc with uniformly continuous derivatives is said to be quasianalytic if the boundary value of a non-zero function from the class can not have a zero of infinite multiplicity. Such classes were described in the 1950-s and 1960-s by Carleson, Rodrigues-Salinas and Korenblum.

A non-zero function from a quasianalytic space of analytic functions can only have a finite number of zeros in the closed disc. Recently, Borichev,  Frank, and Volberg proved an explicit estimate on the number of zeros, for the case of quasianalytic Gevrey classes. Here, an estimate of similar form for  general analytic quasianalytic classes is proved using a reduction to the classical quasianalyticity problem.
\end{abstract}

\section{Introduction}
\paragraph{Analytic quasianalyticity.} 
Let $W = (w_n)_{n =0}^\infty$ be a weight such that
\begin{equation}\label{eq:w}
w_n \in [1, +\infty], \quad \sum_{n= 0}^\infty \frac1{w_n} = 1~, \quad \frac{1}{w_n} = O(n^{-\infty})~. 
\end{equation}
Consider the following space $\mathfrak{A}_W$ of analytic functions in the unit disc $\mathbb D = \{|z| < 1\}$:
\begin{equation}\label{eq:aw}
\mathfrak{A}_W = \left\{ f(z) = \sum_{n=0}^\infty a_n z^n \, \big| \, \|f\|_W \overset{\text{def}}{=} \sup_n |a_n| w_n < \infty \right\}~.
\end{equation}
For each $k$, the $k$-th derivative $f^{(k)}$ of a function $f \in \mathfrak{A}_W$ is uniformly continuous in $\mathbb D$, and hence admits boundary values 
\[ f^{(k)}(e^{i \theta}) = \lim_{z \to e^{i\theta}, \, z \in \mathbb D} f^{(k)}(z)\]
on $\partial \mathbb D$. 

The class $\mathfrak{A}_W$ is said to be quasianalytic if a non-zero function $f \in \mathfrak{A}_W$ can not vanish with all derivatives at a point:
\begin{equation}\label{eq:quas} \forall k \geq 0 \, f^{(k)}(e^{i\theta_0}) = 0 \,\,  \Longrightarrow \,\,  f \equiv 0~, \,\, \text{i.e. } \forall n \geq 0\, a_n = 0~.\end{equation}
A result proved by Carleson \cite{Car}, Rodrigues-Salinas \cite{RodSal} and  Korenblum \cite{Kor} (which we state explicitly in Remark~\ref{rem:D} at the end of this introduction) implies that the condition
\begin{equation}\label{eq:cond*}
\sum_{k=1}^\infty \frac{M_{k-1}}{M_k} = \infty~, \quad \text{where } M_k = \sum_{n=0}^\infty \frac{n^{k/2}}{w_n}~,
\end{equation}
is sufficient for quasianalyticity. If the weights are sufficiently regular, e.g.\ $w_{n-1} \leq w_{n}$ and $w_{2n} \leq \sqrt{w_{n} w_{4n}}$ for $n \geq 1$, the condition (\ref{eq:cond*}) is also necessary for (\ref{eq:quas}).\footnote{In general, the condition (\ref{eq:cond*}) is not necessary. To ensure the quasianalyticity of the class $\mathfrak A_W$, it suffices for the measure $\sum_{n = 0}^\infty w_n^{-1} \delta_n$ to be Stieltjes-determinate; this condition is strictly weaker than (\ref{eq:cond*}).} For such regular weights, the condition (\ref{eq:cond*}) is equivalent to the divergence 
\begin{equation} \sum_{n \geq 0} \frac{\log w_n}{1 + n^{3/2}} = \infty~.\end{equation}
For example, the Gevrey weights
\begin{equation}\label{eq:gev}
W^{(\alpha, a)} = (w^{(\alpha, a)}_n)_{n \geq 0}~, \quad w_n^{(\alpha, a)} = \exp(a n^\alpha + c(\alpha, a))~,\end{equation}
where $c(\alpha, a) = \log \sum_{n \geq 0} \exp(- a n^\alpha)$ is determined by the normalisation (\ref{eq:w}), define a quasianalytic class if and only if $\alpha \geq 1/2$.  

More recently, the problem of analytic quasianalyticity (for the classes $\mathfrak D_M \supset \mathfrak A_W$ as in Remark~\ref{rem:D} below) was studied by Borichev \cite{Bor1}, who obtained a new proof of quasianalyticity in the quasianalytic case (\ref{eq:cond*}) as well as a bound on the growth of $f$ near a zero of infinite multiplicity in the case when (\ref{eq:cond*}) fails.

\paragraph{Zeros in the closed disc, and an application in spectral theory.}
If the space  $\mathfrak A_W$ is quasianalytic, a non-zero function $f \in \mathfrak A_W$ has a finite number of zeros in $\overline{\mathbb D}$, counting multiplicity. Indeed, if $f$ has an infinite number of zeros, these have an accumulation point $e^{i\theta_0} \in \partial \mathbb D$, and then $f$ vanishes with all derivatives at $e^{i\theta_0}$. 

This fact was exploited by Pavlov \cite{Pav1,Pav2} to show that a non-selfadjoint Schr\"odinger operator $H y  = -y''+q(x)y$ with a continuous complex potential $q: \mathbb R_+ \to \mathbb C$, defined on the semiaxis $[0, \infty)$ with the boundary condition $y(0) - hy'(0) = 0$, has a finite number of eigenvalues, counting multiplicity, if
\begin{equation}\label{eq:pavcond} b_k = \int_0^\infty |q(x)| x^{k+1} dx < \infty \text{ for }k \geq 0  \text{ and }  \int_0 \log \inf_k \frac{(\frac{b_{k+1}}{k+1} + b_k)t^k}{k!} dt =- \infty~.\end{equation}
For example, the condition $|q(x)| \leq C \exp(- c x^{\alpha})$ implies (\ref{eq:pavcond}) if and only if $\alpha \geq \frac12$. For $\alpha < \frac12$ Pavlov constructed a potential such $|q(x)| \leq C \exp(- c x^{\alpha})$ but $H$ has infinitely many eigenvalues.

Recently, Bairamov, \c{C}akar and Krall \cite{Bai} and Golinskii and Egorova \cite{GolEg} obtained counterparts of Pavlov's results for non-selfadjoint Jacobi matrices. Consider the operator $J$ acting on $\ell_2(\mathbb Z_+)$ via
\begin{equation}\label{eq:J} (J y) (n)= a_n y(n+1) + b_n y(n) +  \mathbbm{1}_{n \geq 1} c_{n-1} y(n-1)~, \quad n \geq 0~.\end{equation}
It follows from the results of \cite{GolEg} that if 
\begin{equation}\label{eq:pavcond'}  \sum_{k=1}^\infty m_k^{-1/k} = \infty~, \quad \text{where} \quad m_k = \sum_{n =0}^\infty \left( |b_n| + |a_n c_n - 1| \right) n^{k/2}~, \end{equation}
then $J$ has a finite number of eigenvalues, counting multiplicity. The condition (\ref{eq:pavcond'}) holds, for example, when 
\[ |b_n| + |a_n c_n - 1|  \leq C \exp(- c n^{\alpha})\] 
with $\alpha \geq 1/2$, whereas for $\alpha < 1/2$ there exists \cite{GolEg} such  an operator with infinitely many eigenvalues.

\paragraph{Estimates on the number of zeros.}
Denote by $n_f$ the number of zeros of $f$ in $\overline{\mathbb D}$, counting multiplicity, and let 
\begin{equation}\label{eq:NW}N_W(A) = \sup \left\{ n_f \, \big| \, f \in \mathfrak A_W~, \,\, |f(0)| \geq e^{-A} \|f\|_W \right\}~ , \quad A \geq 0~.\end{equation}
A compactness argument shows that $N_W(A)$ is finite for any $A < \infty$. However, it is also of interest to obtain explicit bound on $N_W$, and in particular to investigate the asymptotic behaviour as $A \to +\infty$. Using the method of Pavlov \cite{Pav1,Pav2}, such bounds can be translated into explicit bounds on the number of eigenvalues of the Schr\"odinger operator $H$ as well as of its Jacobi counterpart $J$. 

In view of these applications, Borichev, Frank and Volberg \cite{BFV} proved an explicit bound on $N_W(A)$ for the Gevrey weights (\ref{eq:gev}). Their results  imply that 
\begin{equation}\label{eq:bdgev} N_{W^{(\alpha, a)}} \leq \begin{cases}
C(\alpha, a) A^{\frac{\alpha}{2\alpha-1}}~, & \alpha \in (\frac12, \frac12 + \epsilon] \\
C_1(a) \exp(C_2(a) \sqrt{A})~, &\alpha = \frac12 
\end{cases}~, \end{equation}
with explicit  $C, C_1, C_2$, along with improved bounds for small values of $A$.  The argument of \cite{BFV} is based on the method of pseudoanalytic extension introduced by Dyn$'$kin \cite{Dyn} and applied to analytic quasianalyticity by Borichev in \cite{Bor1}.

Here we employ a reduction to the classical (Hadamard) quasianalyticity problem to prove
\begin{prop}\label{prop} Let $W$ be a weight as in (\ref{eq:w}) satisfying the condition (\ref{eq:cond*}), and let 
\begin{align}
&h(p) = \frac{M_{p-1}}{M_p}~, \,\, p \geq 1;& &H(p) = \sum_{k=1}^p h(k)~;&\\
&h^{-1}(\epsilon) = \min \left\{ p \geq 1 \, \big| \, h(p) \leq \epsilon \right\}~,& &
H^{-1}(R) = \min \left\{ p \geq 0 \, \big|  \, H(p) \geq R \right\}~.
& \end{align}
Then the quantity $N_W(A)$ from (\ref{eq:NW}) satisfies
\begin{equation}\label{eq:mainest}\begin{split} 
N_W(A) &\leq 300 \, h\left( 2 \max (p(A),  h^{-1} (\frac{1}{p(A)}))\right)^{-2}~, 
\text{ where } p(A) = H^{-1}(H(\lceil A+3\rceil) + 25 \sqrt{A})~. \end{split}\end{equation}
\end{prop}
\begin{rem} In our normalisation (\ref{eq:w}), $N_W(A) = 0$ for $A < \log 2$ as a consequence of the Rouch\'e theorem, hence (\ref{eq:mainest}) is meaningful for $A \geq \log 2$.\end{rem}

\begin{rem}
In the Gevrey case (\ref{eq:gev}), 
\[ h(p) \asymp p^{-\frac1{2\alpha}}~, \quad H(p) \asymp \begin{cases} p^{1-\frac{1}{2\alpha}}~, & \frac12 < \alpha \leq 1 \\ \log p~, & p = \frac12 \end{cases}~,\]
hence the bound (\ref{eq:mainest}) implies that
\begin{equation}\label{eq:bdgev'}N_{W^{(\alpha, a)}}(A)\leq \begin{cases}
C'(\alpha, a) A^{\frac{2\alpha}{2\alpha-1}}~, & \alpha \in (\frac12, 1] \\
C'_1(a) \exp(C'_2(a) \sqrt{A})~, &\alpha = \frac12 
\end{cases}~, \end{equation}
which is similar to (\ref{eq:bdgev}), albeit with an inferior exponent for $\alpha > \frac12$. 
\end{rem}

\begin{rem}The estimate (\ref{eq:mainest}) remains valid in the non-quasianalytic situation, provided that $A$ is sufficiently small for the right-hand side to be finite, i.e.\
\begin{equation}\label{eq:nonq} \sum_{k > \lceil A+3 \rceil} h(k) > 25 \sqrt{A}~.\end{equation}
Note that the condition (\ref{eq:nonq}) may hold for large $A$ (particularly, for $A \geq \log 2$) if the series $\sum M_{k-1}/M_k$ converges slowly enough. \end{rem}

\begin{rem}\label{rem:D} Proposition~\ref{prop} also yields bound on the number of zeros of a function in the Carleson--Salinas--Korenblum class
\[ \mathfrak{D}_M = \left\{ f(z) = \sum_{n=0}^\infty a_n z^n \, \big| \, \|f\|_{\mathfrak D_M} \overset{\text{def}}{=} \sup_k \sup_{|z| < 1} \frac{|f^{(k)}(z)|}{M_{2k}} < \infty \right\}\]
associated with a positive sequence $M = (M_k)_{k \geq 0}$. We sketch the (well-known) reduction: first, one may assume without loss of generality that $M_{k} \leq \sqrt{M_{k-1}M_{k+1}}$. The theorem of Carleson--Salinas--Korenblum asserts that in this case $\mathfrak{D}_M$ is quasianalytic if and only if 
\begin{equation}\label{eq:quas2} \sum_{k \geq 0} \frac{M_{k-1}}{M_k} = \infty~. \end{equation}
Construct the weight
%\[ W(M)= (w_n)_{n \geq 0}~, \,\, w_n =  \left[ \sum_{m \geq 0} \max_{0 \leq k \leq m} \frac{m(m-1)\cdots (m-k+1)}{M_{2k}} \right]^{-1} \max_{0 \leq k \leq n} \frac{n(n-1)\cdots (n-k+1)}{M_{2k}}\]
\[ W(M)= (w_n)_{n \geq 0}~, \,\, w_n =  \frac{\tilde w_n}{\sum_{m=0}^\infty \tilde w_m}~, \,\, \text{where }\tilde w_m =  \max_{0 \leq k \leq m} \frac{m(m-1)\cdots (m-k+1)}{M_{2k}}\]
so that $\mathfrak D_M \subset \mathfrak A_{W(M)}$. One can check that if (\ref{eq:quas2}) holds, then also 
\[ M^1 = (M^1_k)_{k \geq 0}~, \quad M_k^1 = \sum_{n \geq 0} \frac{n^{k/2}}{w_n} \]
satisfies $\sum_{k \geq 0} M_{k-1}^1 / M_{k}^1 = \infty$. Therefore Proposition~\ref{prop} applied to $W(M)$ yields an estimate on 
\begin{equation}
\label{eq:NM}N_{\mathfrak D_M}(A) = \sup \left\{ n_f \, \big| \, f \in \mathfrak D_M~, \,\, |f(0)| \geq e^{-A} \|f\|_{\mathfrak D_M} \right\}~ , \quad A \geq 0~,\end{equation}
for an arbitrary quasianalytic $\mathfrak D_M$.
\end{rem}

\section{Proof of Proposition~\ref{prop}} The proof is based on the following construction, similar to the one using which the determinacy criteria for the moment problem in the Stieltjes case are derived from those in the Hamburger case (see \cite{me} for a further application of a similar construction). To every 
\[ f(z) = \sum_{n=0}^\infty a_n z^n \in \mathfrak A_W \]
we associate a function
\[ \phi_f(x) = \sum_{n=0}^\infty a_n \cos (\sqrt{n} x)~, \quad x \in \mathbb R~. \]
We have: 
\[|\phi_f^{(k)}(x)| \leq \sum_{n=0}^\infty |a_n| n^{k/2} \leq \|f\|_W M_k~, \]
i.e.\ $\phi_f$ lies in the space 
\[ \mathfrak Q_M = \left\{ \phi \in C^\infty(\mathbb R) \, \big| \, \|\phi\|_{\mathfrak Q_M} \overset{\text{def}}{=} \sup_k \frac{\|\phi^{(k)}\|_\infty}{M_k} < \infty \right\} \]
defined by the sequence $M = (M_k)_{k \geq 0}$ of (\ref{eq:cond*}). According to the Denjoy--Carleman theorem in the form of Mandelbrojt (see \cite{Bang} or \cite{Man}, and also the comment following Lemma~\ref{l:bang} below), the condition $\sum_{k=1}^\infty M_{k-1}/M_k = \infty$ implies that the class $\mathfrak Q_M$ is quasianalytic.\footnote{In our case, the sequence $M$ is logarithmically convex, i.e.\  $M_k \leq \sqrt{M_{k+1} M_{k-1}}$ for $k \geq 1$, hence the condition $\sum_{k=1}^\infty M_{k-1}/M_k = \infty$ is necessary and sufficient for the quasianalyticity of $\mathfrak Q_M$.}  This  implies  the sufficiency part of the Carleson--Salinas--Korenblum condition (\ref{eq:cond*}) for the quasianalyticity of $\mathfrak A_W$: indeed, if $f$ vanishes with all derivatives at $1$, then $\phi_f$ vanishes with all derivatives at $0$, and hence $\phi_f$ and $f$ are identically zero.

\medskip
To prove Proposition~\ref{prop}, we make these considerations quantitative. The argument rests on two lemmas. The first one asserts that $\phi_f$ and its first few derivatives are small at $0$ if $f$ has many zeros near $1$. 
\begin{lemma}\label{l:1} Let $\epsilon \in (0, 1)$, and let $m$ be the number of zeros of $f \in \mathfrak A_W$ in the domain $\{|z| \leq 1~, |z -1| < \epsilon\}$, counted with multiplicity. Then
\[ |\phi_f^{(2k)}(0)| \leq \left( \frac{4e\epsilon}m \right)^{m-k} M_{2m} \|f \|_W~, \quad 0 \leq k \leq \min(\frac{m}{2}, \sqrt{\frac{m}{8\epsilon}})~. \]
\end{lemma}
The second lemma guarantees that there is a point not too far from $0$ at which $\phi_f$ is not too small. The current version, with the sharp power of $A$, was kindly communicated by F.~Nazarov.
\begin{lemma}\label{l:2}
Let $\phi(x) = \sum_{n=0}^\infty a_n \cos (\sqrt n x)$ be such that $|a_0| \geq e^{-A}$  and $\sum |a_n| \leq 1$. Then there exists $x \in [0, 9 \sqrt A]$ such that $|\phi(x)| \geq e^{-A-3}$.
\end{lemma}
To derive the proposition from the two lemmas, we use a propagation of smallness argument due to Bang \cite{Bang}, which we state as 
\begin{lemma}\label{l:bang}
Let $M = (M_k)_{k \geq 0}$ be a sequence of positive numbers such that $M_k \leq \sqrt{M_{k-1} M_{k+1}}$ for $k \geq 1$. For $\phi \in \mathfrak Q_M$, define
a nested sequence of sets $\mathbb R = B_0(\phi) \supset B_1(\phi) \supset B_2(\phi) \supset \cdots$ via
\[ B_p(\phi) = \left\{ x \in \mathbb R \, \big| \, \forall 0 \leq k < p \, |\phi^{(k)}(x)| \leq e^{k-p} M_k \|\phi\|_{\mathfrak Q_M}~\right\}~. \]
Then for $0 \leq q < p$
\[ \operatorname{dist}(B_p(\phi), \mathbb R \setminus B_q(\phi)) \geq \frac{1}{e} (H(p) - H(q)) = \frac{1}{e} \sum_{k=q+1}^p \frac{M_{k-1}}{M_k}~.\]
\end{lemma}
(As pointed out in \cite{Bang}, this lemma readily implies the Denjoy--Carleman theorem mentioned above.)  The proofs of the lemmas are postponed to the next section, and we now proceed to

\begin{proof}[Proof of Proposition~\ref{prop}]Without loss of generality we may assume that $\|f\|_W = 1$, so that $\|\phi_f\|_{\mathfrak Q_M} \leq 1$. Denote by $n_f(S)$ the number of zeros of $f$ in $S \subset \overline{ \mathbb D}$, counting multiplicity. By Jensen's formula
\[\begin{split} -A = \log |f(0)| 
&= \sum_{f(z) = 0} \log |z| + \int_0^{2\pi} \log |f(e^{i\theta})| \frac{d\theta}{2\pi}\\ 
&\leq \sum_{f(z) = 0}  \mathbbm{1}_{|z| < 1 - \frac\epsilon2}\log |z|\\
&\leq\log(1-\frac\epsilon2) \, n_f(\{|z| < 1-\frac\epsilon2\})~,\end{split} \]
hence
\begin{equation}\label{eq:jenbd}n_f \leq \frac{2}{\epsilon} A + n_f(\{|z| \geq 1 - \frac\epsilon 2\}) \leq \frac{1}{\epsilon} (2 A + 8 m_\epsilon)~, \end{equation}
where 
\[ m_\epsilon = \sup_\theta n_f(\{|z| \leq 1~, \, |z - e^{i\theta}| < \epsilon\})~. \]
Without loss of generality the supremum in the definition of $m_\epsilon$ is achieved when $\theta =0$. 

Let $p(A) = H^{-1}(H(\lceil A+3 \rceil) + 25\sqrt{A})$, and let 
\[ m = \max(p(A), h^{-1}(\frac{1}{p(A)}))~, \quad
\epsilon = \frac{1}{4e^2} \frac{mM_{2m-1}^2}{M_{2m}^2}~, \]
so that
\[ \sqrt\frac{m}{2\epsilon} = \sqrt{2} e \frac{M_{2m}}{M_{2m-1}} \geq \frac{M_m}{M_{m-1}} = \frac{1}{h(m)} \geq p(A)~. \] 
Let us show that $m_\epsilon < m$. Assume the contrary. Observe that 
\[ \frac{p(A)}{2} \leq \min (\frac m 2, \sqrt{\frac{m}{8 \epsilon}})~,\]
therefore Lemma~\ref{l:1} yields  
\begin{equation} |\phi_f^{(2k)}(0)| \leq \left( \frac{4e\epsilon}m \right)^{m-k} M_{2m}\end{equation}
for $0 \leq k \leq p(A)/2$. Estimating
\[ \frac{\left(\frac{4e\epsilon}{m}\right)^{l-1} M_{2l-2}}{\left(\frac{4e\epsilon}{m}\right)^{l} M_{2l}} = e \frac{M_{2l-2}}{M_{2l}} \frac{M_{2m}^2}{M_{2m-1}^2} \geq e~, \quad 0 \leq l \leq m~,  \]
we obtain that 
\[ \left( \frac{4e\epsilon}m \right)^{m} M_{2m} 
=  \left( \frac{4e\epsilon}m \right)^{k} M_{2k} \prod_{l=k+1}^m \frac{\left(\frac{4e\epsilon}{m}\right)^{l} M_{2l}}{\left(\frac{4e\epsilon}{m}\right)^{l-1} M_{2l-2}}
\leq e^{-(m-k)}   \left( \frac{4e\epsilon}m \right)^{k} M_{2k}~. \] 
Therefore (using that $3k \leq 3p(A)/2 \leq m + p(A)$)
\[  |\phi_f^{(2k)}(0)|  \leq e^{-(m-k)} M_{2k} \leq e^{-(2k - p(A))} M_{2k}~.\]
Trivially, $\phi_f^{(2k+1)}(0) = 0$ for all $k$. Hence $0 \in B_{p(A)}(\phi_f)$. On the other hand, by Lemma~\ref{l:2}, 
\[ \operatorname{dist}(0, \mathbb R \setminus B_{\lceil A+3 \rceil}(\phi_f)) \leq 9 e  \sqrt{A}~. \]
Applying Lemma~\ref{l:bang}, we deduce that 
\[ H(p(A)) - H(\lceil A+3 \rceil) \leq 9e \sqrt{A} < 25 \sqrt{A}~,\]
in contradiction with the definition of $p(A)$. This completes the proof of the estimate $m_\epsilon < m$. 

Returning to (\ref{eq:jenbd}) and recalling that $m \geq p(A) \geq A$, we obtain:
\[ n_f \leq  \frac{1}{\epsilon} (2 A + 8 m_\epsilon) \leq \frac{10m}{\epsilon} \leq 300 h(2m)^{-2}~.\qedhere\]
\end{proof}

\section{Proofs of the lemmas}
In the proof of the Lemma~\ref{l:1}, we use the following  lemma which is borrowed from the work of M.~Lavie \cite[Lemma 3]{Lav}.
\begin{lemma}\label{l:lag}
Let $R \subset \mathbb C$ be a closed convex set of diameter $\delta$. If $f(z)$ is analytic in $R$ and vanishes at $m$ points of $R$ (counting multiplicity), then 
\begin{equation}\label{eq:lavie} \max_{z \in R} |f^{(k)}(z)| \leq \frac{\delta^{m-k}}{(m-k)!} \max_{z \in R} |f^{(m)}(z)|~, \quad 0 \leq k \leq m~. \end{equation}
\end{lemma}
In \cite{Lav}, this inequality is proved by induction, using the formula 
\[ \frac{d^k}{dz^k} \frac{f(z)}{\alpha - z} = (\alpha - z)^{-k-1} \int_\alpha^z \frac{f^{(k+1)}(\zeta)}{(\alpha - \zeta)^k} d\zeta~,\]
valid if $f(\alpha) = 0$. As mentioned in \cite{Lav},  (\ref{eq:lavie}) can be also proved using the Hermite formula for divided differences.
\begin{rem}\label{r:lag}
By an approximation argument, the conditions of the lemma can be relaxed as follows:
(a) if $R = \overline {\operatorname{int} R}$, then the lemma remains valid if instead of assuming that $f$ is analytic in $R$, we assume  that $f$ analytic in $\operatorname{int} R$ and that $f, f', f',\cdots, f^{(m)}$ are uniformly continuous in $R$. (b) If $R \subset \mathbb R$, it suffices to assume that $f \in C^m(R)$. 
\end{rem}

\begin{proof}[Proof of Lemma~\ref{l:1}]
Recall (see \cite{RD}) that the Stirling numbers of the second kind are defined via
\[ \stirling{k}{l} = \frac{1}{l!} \sum_{j=0}^l (-1)^{l-j} \binom{l}{j} j^k~,\]
so that
\[ n^k = \sum_{l=0}^k \stirling{k}{l} \, n(n-1)\cdots (n-l+1)~,\]
and that 
\[ 0 \leq \stirling{k}{l} \leq \frac12 \binom{k}{l} l^{k-l} \leq \frac12 k^{2(k-l)}~.\]
Then
\[\begin{split} |\phi_f^{(2k)}(0)| &= \left| \sum_{n \geq 0} a_n n^k \right|
\leq \sum_{l=0}^k \stirling{k}{l} \left| \sum_{n \geq 0} a_n n (n-1) \cdots (n-l+1) \right| \\
&= \sum_{l=0}^k \stirling{k}{l} |f^{(l)}(1)|
\leq \frac12 \sum_{l=0}^k k^{2(k-l)}  |f^{(l)}(1)|~.
\end{split}\]
By Lemma~\ref{l:lag} and the subsequent Remark~\ref{r:lag}, we have:
\[ |f^{(l)}(1)| \leq \frac{(2\epsilon)^{m-l}}{(m-l)!}  M_{2m}~,\]
hence
\[\begin{split} |\phi_f^{(2k)}(0)| &\leq \frac12 \sum_{l=0}^k k^{2(k-l)}  \frac{(2\epsilon)^{m-l}}{(m-l)!}  M_{2m}\\
&\leq \frac12 \frac{(2\epsilon)^{m-k}}{(m-k)!} M_{2m} \sum_{l=0}^k \left(\frac{2k^2 \epsilon}{m-k}\right)^{k-l} \\
&= \frac12 \frac{(2\epsilon)^{m-k}}{(m-k)!} M_{2m} \sum_{l=0}^k \left(\frac{4k^2 \epsilon}{m}\right)^{k-l} \leq \left(\frac{4e\epsilon}{m}\right)^{m-k} M_{2m}~, 
\end{split}\]
provided that $m \geq \max(2k, 8k^2\epsilon)$.
\end{proof}

\begin{proof}[Proof of Lemma~\ref{l:2} (F.~Nazarov)]
Define a sequence of independent random variables $X_j$ so that $X_j \sim \operatorname{Unif}[-\frac{\pi}{\sqrt{j}}, \frac{\pi}{\sqrt{j}}]$, and let $S_N = X_1 + \cdots + X_N$. Then 
\[ |S_N| \leq \pi \sum_{j=1}^N \frac{1}{\sqrt{j}} \leq 2\pi \sqrt{N} \]
and 
\[ g_N(\xi) = \mathbb E \cos(\xi S_N) = \mathbb E \exp(i \xi S_N) = \prod_{j=1}^N \frac{\sin \frac{\pi \xi}{\sqrt{j}}}{\frac{\pi \xi}{\sqrt{j}}}~. \]
Therefore 
\[ \mathbb E \phi(S_N) = \sum_{n \geq 0} a_n g_N(\sqrt{n}) = a_0 + \sum_{n \geq N+1} a_n g_N(\sqrt n)~. \] 
Now, for $\xi \geq \sqrt N$
\[ |g_N(\xi)| \leq \prod_{j=1}^N \frac{\sqrt j}{\pi \xi} \leq \prod_{j=1}^N \frac{1}{\pi} = \pi^{-N}~, \]
therefore
\[ |\mathbb E \phi(S_N)| \geq |a_0| - \sum_{n \geq N+1} \pi^{-N} |a_n| \geq e^{-A} - \pi^{-N}~. \]
Letting $N = \lceil A \rceil$, we obtain that there exists $x \in [0,  2\pi \sqrt{\lceil A \rceil}]$ such that \[ |\phi(x)| \geq e^{-A} ( 1 - e/\pi) \geq e^{-A - 3}~.\] 
For $A \geq \frac12$, $2\pi \sqrt{\lceil A \rceil}  \leq 9 \sqrt{A}$, as claimed. For $A < \frac 12$,
\[ |\phi(0)| \geq e^{-1/2} - (1 - e^{-1/2}) \geq e^{-3} \geq e^{-A-3}~.\qedhere\] 
\end{proof}

\begin{proof}[Proof of Lemma~\ref{l:bang}] We reproduce the original argument of Bang \cite{Bang}. It suffices to show that if $x \in B_p(\phi)$ and $h = |y - x| \leq \frac{1}{e} \frac{M_{p-1}}{M_p}$, then $y \in B_{p-1}(\phi)$. Expanding $\phi$ in a Taylor series, we have for $0 \leq k < p-1$:
\[\begin{split} 
|\phi^{(k)}(y)| &\leq \sum_{j=0}^{p-k-1} |\phi^{(k+j)}(x)| \frac{h^j}{j!}  + |\phi^{(k+j)}(y_1)| \frac{h^{p-k}}{(p-k)!} \\
&\leq \sum_{j=0}^{p-k} e^{k+j-p} M_{k+j} \|\phi\|_{\mathfrak Q_M} \frac{h^j}{j!} ~. \end{split}\]
Now we bound $M_{k+j} \leq M_k (M_{p}/M_{p-1})^{j}$ and obtain:
\[\begin{split} |\phi^{(k)}(y)| &\leq e^{k-p} M_k \|\phi\|_{\mathfrak Q_M} \sum_{j=0}^{p-k-1} \frac{e^j (M_{p}/M_{p-1})^j h^j}{j!} \\
&\leq e^{k-p} M_k \|\phi\|_{\mathfrak Q_M}  e^{eh \frac{M_{p}}{M_{p-1}}}  \leq e^{k-p+1} M_k \|\phi\|_{\mathfrak Q_M}~.\qedhere
\end{split}\]
\end{proof}

\paragraph{Acknowledgement} I am grateful to A.\ Borichev, M.\ Sodin, J.\ Stoyanov, and  A.\ Volberg for helpful comments, and to F.~Nazarov for explaining me the proof of the current Lemma~\ref{l:2}.

\end{document}